\documentclass[11pt,a4paper]{article}
\usepackage[utf8]{inputenc}
\usepackage[T1]{fontenc}
\usepackage{amsfonts,amssymb,amsmath,amsthm}
\usepackage[margin=2.5cm]{geometry}
\usepackage{hyperref}
\hypersetup{hidelinks}

\newcommand{\R}{\mathbb{R}}
\newcommand{\Rn}{\mathbb{R}^n}

\newcommand{\Om}{\Omega}

\newcommand{\test}{C_c^1}
\newcommand{\prt}{\partial}

\newcommand{\dd}{\mathrm{d}}
\newcommand{\Ln}{\mathcal{L}^n}
\newcommand{\Hnmo}{\mathcal{H}^{n-1}}
\newcommand{\e}{\varepsilon}

\newcommand{\al}{\alpha}
\DeclareMathOperator{\Div}{div}
\DeclareMathOperator{\diam}{diam}
\newcommand{\sgn}{\operatorname{sgn}}

\newcommand{\uw}{{\mathcal{W}}}

\newcommand{\Fo}{F_{\circ}}
\def\aver{\mathop{\,\rlap{---}\!\int}\nolimits}

\theoremstyle{plain}
\newtheorem{theorem}{Theorem}[section]
\newtheorem{proposition}[theorem]{Proposition}
\newtheorem{lemma}[theorem]{Lemma}

\theoremstyle{definition}

\theoremstyle{remark}

\numberwithin{equation}{section}

\title{\bf Optimal Non-Homogeneous Improvements of the Anisotropic Hardy Inequality}
\author{Giuseppina di Blasio \and Giovanni Pisante \and Georgios Psaradakis}
\date{}

\begin{document}
\maketitle

\begin{abstract}
\noindent We consider the Hardy inequality associated with a strongly convex Finsler
norm $F$ on $\Rn$ and we obtain its optimal non-homogeneous improvements in the whole
range $1<p<\infty$. We show that the anisotropic Hardy difference controls a weighted
critical Sobolev norm when $1<p<n$, a weighted exponential integrability of Trudinger
type when $p=n$, where the Hardy term is replaced by its Leray logarithmic correction,
and a weighted H\"older seminorm when $p>n$, proving in each regime that the
logarithmic weights appear with the optimal exponents. In the borderline case $p=n$
this extends to general functions the optimal exponent previously known only for
$\Fo$-radial functions. The proofs rely on a common ground-state reduction and, for $1<p\le n$, on a
splitting of the transformed function into its $\Fo$-spherical mean and a
mean-free part. Since $F^2$ is not a quadratic form unless $F$ is induced by a
scalar product, the reduction of the radial part through spherical harmonics,
used in the Euclidean case, is not available; we control the radial part
directly, through a convexity inequality for $F^p$ with a quadratic remainder
and sharp one-dimensional estimates, and the mean-free part through a
Poincar\'e inequality on the boundary of the Wulff shape. For $p>n$ we use an
anisotropic Morrey representation formula.
\end{abstract}



\section{Introduction}

In this paper we deal with the Hardy inequality associated with a general norm $F$ of
$\Rn$ and we study its optimal non-homogeneous improvements in the whole range
$1<p<\infty$. Throughout the paper, $F:\Rn\to[0,\infty)$, $n\ge2$, is an even,
positively $1$-homogeneous and strongly convex norm (see \S\ref{ssec:finsler} for the
precise assumptions), and $\Fo$ denotes its polar norm. When $F$ is the Euclidean norm
we recover the classical setting. Norms of this type arise naturally in the study of the
anisotropic $p$-Laplacian and of the Wulff shape \cite{BFK}, and in anisotropic geometric
problems such as the motion by anisotropic mean curvature \cite{BP}.

Let $\Om\subset\Rn$ be a bounded domain containing the origin and set
$R_\Om:=\sup_{x\in\Om}\Fo(x)$. The anisotropic $L^p$-Hardy inequality asserts that
\begin{equation}\label{anisoHardyintrod}
 I_{p,F}[u]:=\int_\Om F^p(\nabla u)\,\dd x
 -\Big|\frac{n-p}{p}\Big|^p\int_\Om\frac{|u|^p}{\Fo^{\,p}}\,\dd x\ \ge\ 0,
\end{equation}
for every $u\in\test(\Om)$ if $1<p<n$, and for every $u\in\test(\Om\setminus\{0\})$ if
$p>n$; the constant $|(n-p)/p|^p$ is sharp and it is not attained. In the anisotropic setting, for $1<p<n$ inequality \eqref{anisoHardyintrod} can be deduced from the P\'olya--Szeg\H{o} principle for convex symmetrization \cite{AFLT} (see also \cite{VS}), while a sharp form valid for every $p\neq n$, with the smaller directional integrand
$|\nabla u\cdot x/\Fo|^p$ in place of $F^p(\nabla u)$, is due to Mercaldo, Sano and
Takahashi \cite[Theorem~1.1]{MST}; see also \cite{BlPPs}, \cite{PBG} for Hardy
inequalities involving the anisotropic distance to the boundary, and \cite{Ho} for general potentials.

In the borderline case $p=n$ the Hardy constant vanishes, and the inequality has to be
modified. Already in 1933, Leray \cite{Le} proved in the plane a substitute inequality,
in which the Hardy term carries a logarithmic correction; its multidimensional version,
with best constant $((n-1)/n)^n$ and optimal power of the logarithmic weight, is
contained in \cite{BFT1}. Setting
\[
 X_1(t):=(1-\log t)^{-1}\quad (t\in(0,1]),\qquad X_1(0):=0,
\]
the anisotropic counterpart reads
\begin{equation}\label{anisoLeray}
 I_{n,F}[u]:=\int_\Om F^n(\nabla u)\,\dd x
 -\Big(\frac{n-1}{n}\Big)^{\!n}\int_\Om\frac{|u|^n}{\Fo^{\,n}}\,
 X_1^n\Big(\frac{\Fo}{R_\Om}\Big)\,\dd x\ \ge\ 0,
 \qquad u\in\test(\Om),
\end{equation}
and it follows from \cite[Theorem~1.2]{MST}, where the larger weight
$\log^{-n}(R_\Om/\Fo)$ appears in place of $X_1^n(\Fo/R_\Om)$.

It is worth noticing that the three forms of the Hardy inequality mirror the three
regimes of the Sobolev embedding of $W^{1,p}_0(\Om)$. For $1<p<n$ a point has zero
$p$-capacity, $W^{1,p}_0(\Om)$ embeds into $L^{p^\ast}(\Om)$, $p^\ast=np/(n-p)$, and the
singularity of the Hardy weight at the origin is harmless, so that \eqref{anisoHardyintrod} holds on the whole $\test(\Om)$. At $p=n$ the Sobolev exponent blows up and is replaced by the exponential integrability of Trudinger \cite{Tr} (see \cite{M} for the sharp constant); at the same time the Hardy constant degenerates, and the logarithmic correction of Leray appears. Finally, for $p>n$ a point has positive $p$-capacity and $W^{1,p}$-functions are H\"older continuous with exponent $1-n/p$ by Morrey's inequality; accordingly, \eqref{anisoHardyintrod} is meaningful only for functions vanishing at the origin.
The same three regimes occur for the anisotropic Sobolev \cite{AFLT,BlPPs}, Trudinger--Moser \cite{ZhZh} and Morrey \cite{KMM} inequalities.

Since the best constant in \eqref{anisoHardyintrod} is not attained, one is naturally
led to ask which kind of positive remainder is controlled by the Hardy difference
$I_{p,F}[u]$. The study of such improved Hardy inequalities goes back to the
contribution of Brezis and V\'azquez \cite{BV}, who proved that in a bounded domain the
Euclidean Hardy difference for $p=2$ controls a multiple of the $L^2$-norm, and has
since been widely investigated (see \cite{ACR,BFT1,FT,VzZ} and references therein).
Roughly speaking, two kinds of remainders can be considered. The \emph{homogeneous} ones
have the same structure as the Hardy term, and lead to the series expansion of the
Hardy inequality with iterated logarithmic weights \cite{BFT1,FT}. The
\emph{non-homogeneous} ones measure instead a gain of integrability or regularity of
$u$, and their form is dictated by the Sobolev embedding. In the Euclidean case the
optimal non-homogeneous improvements are known in all three regimes. For $p=2<n$,
Filippas and Tertikas \cite{FT} proved that the Hardy difference controls a critical
Sobolev norm with the optimal weight $X_1^{1+n/(n-2)}$, and Gkikas and Psaradakis
\cite{GkPs} extended this result to every $1<p<n$, the optimal weight being
$X_1^{1+p^\ast/p}$. For $p=n$, after the contributions of Psaradakis and Spector
\cite{PsSp} and of Mallick and Tintarev \cite{MT} (see also \cite{CR,CR2} for radial
functions), the optimal Leray--Trudinger inequality, with the weight $X_2^{1/n}$, where
$X_2:=X_1\circ X_1$, has been obtained in \cite{BlPPsIUMJ}.
In the Euclidean setting, Wang and Sun \cite{WangSun2026} have extended the
optimal Leray--Trudinger inequality to every finite level of the logarithmic
series. For $p>n$, Psaradakis \cite{Ps} proved the optimal
Hardy--Morrey inequality, in which the H\"older seminorm carries the weight
$X_1^{1/p}$; see also \cite{FPs,HLL} for the analogous problem with the distance to the
boundary, and \cite{Br,HL} for related results on Morrey's inequality.

In the anisotropic setting much less is known. For $p=2$, Krist\'aly and Szak\'al
\cite{KS} proved a Brezis--V\'azquez type inequality in Minkowski spaces, in which
the anisotropic Hardy term with distance to the origin is combined with an
$L^2$-remainder. However, the Hardy constant there is
$\frac{(n-2)^2}{4}-\alpha^2$ with $\alpha\ge\frac{n-2}{2}\sqrt{1-l_F^2}$,
where $l_F\in(0,1]$ is the uniformity constant of the norm, so the sharp
Hardy constant is recovered only when $l_F=1$, that is, when $F$ is induced
by a scalar product. Brezis--V\'azquez type improvements have also been obtained
on Finsler manifolds under curvature assumptions \cite{YZS, Z}, which exclude
the flat case considered here. In all these results the remainder is an $L^2$-term,
and no gain of integrability or regularity beyond the energy space is obtained.
A weighted anisotropic Sobolev inequality, with applications to Hardy inequalities,
has been established in \cite{BlPPs}, and sharp anisotropic Caffarelli--Kohn--Nirenberg
inequalities have been studied in \cite{ShenCKN}. In \cite{BlPPsNA} we proved an anisotropic improved Leray--Trudinger inequality: the optimal exponent $1/n$ on $X_2$ was reached for
$\Fo$-radial functions, whereas for general functions the method gave only the exponent
$2/n$. To our knowledge, remainders of Sobolev or Morrey type in the anisotropic Hardy
inequality~\eqref{anisoHardyintrod} with sharp constant, optimal or not, are not available
in the literature for $p\neq n$, while for $p=n$ the optimal exponent was known
only for $\Fo$-radial functions \cite{BlPPsNA}.

Our results show that the Euclidean picture does not depend on the quadratic
structure of the Euclidean norm. For every strongly convex norm $F$, the Hardy difference controls the
same Sobolev, Leray--Trudinger and Morrey type quantities, with the same optimal
logarithmic exponents, in each of the three regimes. The optimality of these
exponents is a one-dimensional phenomenon, already visible on $\Fo$-radial
functions, for which the problem does not depend on $F$. The content of our
theorems is that no non-radial function behaves worse, a fact whose Euclidean
proof uses this quadratic structure, through an expansion in spherical
harmonics. In the
subcritical range our main result is the following.

\begin{theorem}\label{main}
Let \(n\ge2\), \(1<p<n\), and let \(\Om\subset\Rn\) be a bounded domain
containing the origin. There exists a positive constant \(C=C(n,p,F)\) such that for every
\(u\in\test(\Om)\),
\begin{equation}\label{HS}
 I_{p,F}[u]
 \ge
 C\Bigg(
 \int_\Om |u|^{p^\ast}
 X_1^{\,1+p^\ast/p}\!
 \Big(\frac{\Fo(x)}{R_\Om}\Big)\,\dd x
 \Bigg)^{p/p^\ast},
 \qquad
 p^\ast:=\frac{np}{n-p}.
\end{equation}
Moreover, the power \(1+p^\ast/p\) on \(X_1\) cannot be decreased.
\end{theorem}

Theorem~\ref{main} proves that the Hardy difference controls a critical Sobolev norm,
with a logarithmic weight which vanishes at the origin and whose exponent is optimal.
When \(F=|\cdot|\), it reduces to the case $k=0$ of \cite[Theorem~A]{GkPs}.

In the borderline case $p=n$ we set
\[
 H_{n,F}(\Om):=\big\{u\in\test(\Om)\,\big|\,I_{n,F}[u]\le1\big\},
\]
and we prove the following optimal Leray--Trudinger inequality.

\begin{theorem}\label{anis-main-theorem}
Let \(\Om\subset\Rn\) be a bounded domain containing the origin. Then there exists
\(c(n,F)>0\), depending only on \(n\) and \(F\), such that, for every
\(\lambda<c(n,F)\),
\[
 \sup_{u\in H_{n,F}(\Om)}
 \aver_\Om
 \exp\!\left\{
 \lambda
 \left[
 |u(x)|
 X_2^{1/n}\!
 \left(\frac{\Fo(x)}{R_\Om}\right)
 \right]^{n/(n-1)}
 \right\}
 \,\dd x
 <\infty .
\]
Moreover, the exponent \(1/n\) on \(X_2\) cannot be decreased.
\end{theorem}

As for Trudinger's inequality, Theorem~\ref{anis-main-theorem} is a consequence of a
growth estimate for $L^q$-norms, namely of the critical estimate
\begin{equation}\label{sharp Leray-Trudinger growth}
 \Bigg(
 \aver_\Om
 \left[
 |u|
 X_2^{1/n}\!
 \left(\frac{\Fo}{R_\Om}\right)
 \right]^q
 \,\dd x
 \Bigg)^{1/q}
 \le
 C(n,F)\Big(\frac{|\uw_{R_\Om}|}{|\Om|}\Big)^{1/q}\,q^{1-1/n}\,\big(I_{n,F}[u]\big)^{1/n},
\end{equation}
valid for every $u\in\test(\Om)$ and every $q>n$, where $\uw_{R_\Om}:=\{\Fo<R_\Om\}$;
estimate \eqref{sharp Leray-Trudinger growth} is
the Hardy--Leray analogue of the classical Trudinger growth estimate. When
\(F=|\cdot|\), Theorem~\ref{anis-main-theorem} reduces to the Euclidean result of
\cite{BlPPsIUMJ}. In the anisotropic setting, the optimal exponent \(1/n\) was known
only for \(\Fo\)-radial functions \cite[Theorem~3.2]{BlPPsNA}, whereas for arbitrary
functions \cite[Theorem~3.1]{BlPPsNA} yields the exponent \(2/n\). Thus
Theorem~\ref{anis-main-theorem} extends the optimal logarithmic exponent to general
functions.

Finally, in the supercritical range $p>n$ we set
\[
 \al:=1-\frac np,
 \qquad
 \diam_F(\Om):=\sup_{x,y\in\Om}\Fo(x-y),
\]
and we prove the anisotropic counterpart of \cite[Theorem~B]{Ps}.

\begin{theorem}[optimal anisotropic Hardy--Morrey inequality]
\label{pn-thm-main}
Let \(n\ge2\), \(p>n\), and let \(\Om\subset\Rn\) be a bounded domain
containing the origin. There exists a positive constant
\(C=C(n,p,F)\) such that for every
\(u\in\test(\Om\setminus\{0\})\),
\begin{equation}\label{pn-HM}
 \sup_{\substack{x,y\in\Om\\ x\ne y}}
 \left\{
 \frac{|u(x)-u(y)|}{\Fo^{\,\al}(x-y)}
 X_1^{1/p}\!
 \Big(\frac{\Fo(x-y)}{\diam_F(\Om)}\Big)
 \right\}
 \le
 C\,\big(I_{p,F}[u]\big)^{1/p}.
\end{equation}
Moreover, the weight \(X_1^{1/p}\) is optimal: the exponent \(1/p\)
cannot be decreased.
\end{theorem}

In proving the optimality we also obtain an $X_1^2$-improvement of the anisotropic
$L^p$-Hardy inequality with sharp exponent on the logarithmic weight
(Theorem~\ref{pn-thm-X2}), which is, in our
opinion, of independent interest.

The three regimes share a common ground-state mechanism: writing $u=\phi\,w$,
where $\phi=\Fo^{-(n-p)/p}$ is the ground state of the anisotropic Hardy
operator, a pointwise convexity inequality for $F^p$ and the vanishing of a
cross term provide lower bounds for $I_{p,F}[u]$ in terms of weighted energies
of $w$. We stress that, although the statements are the exact anisotropic
counterparts of the Euclidean ones, the Euclidean proofs do not carry over to
a general Finsler norm, and several new arguments are needed.

As in the Euclidean case \cite{GkPs,BlPPsIUMJ}, for $1<p\le n$ we split the
transformed function into its $\Fo$-spherical mean and a mean-free part. The
mean-free part is controlled by a Poincar\'e inequality, now on $\prt\uw$ with
respect to the measure $\mu$, combined with the anisotropic Sobolev inequality
applied to a suitable auxiliary function. The main obstruction concerns the
radial part when $p>2$. In the Euclidean case \cite[\S4]{GkPs} this part is
reached through a lower bound by an $L^2$-Hardy functional, which is then
reduced to radial functions by an expansion in spherical harmonics; this relies
on the fact that $|\cdot|^2$ is a quadratic form, so that the $L^2$-Hardy
operator is linear and its angular part is a multiple of the Laplace--Beltrami
operator on the sphere. For a general norm $F$, instead, $F^2$ is a quadratic
form only when $F$ is induced by a scalar product, and no such reduction is
available. We bypass it by working directly with the radial derivative of the
transformed function: a convexity inequality for $F^p$ with a quadratic
remainder term, which we prove for every strongly convex norm, together with
Jensen's inequality on $\prt\uw$, reduces the radial part to sharp
one-dimensional weighted inequalities.

Further difficulties arise in each regime. The convexity inequalities have to
be proved for $F^p$ rather than for the Euclidean norm. In the borderline case
$p=n$ the doubly logarithmic weight is carried entirely by the radial part, which is
handled through the sharp radial estimates of \cite{BlPPsNA}, while the mean-free part,
controlled through the Poincar\'e inequality on $\prt\uw$, satisfies an unweighted
Trudinger estimate; it is this splitting that allows us to reach the optimal
exponent $1/n$ for general functions instead of the exponent $2/n$ obtained in
\cite{BlPPsNA}. Finally, for $p>n$ the ground-state reduction is coupled with an
anisotropic Morrey representation formula and with Theorem~\ref{pn-thm-X2}.

The paper is organized as follows. In Section~\ref{sec:prelim} we fix the notation and
recall some preliminary results, including the ground-state reduction. Sections~\ref{sec:subcritical},
\ref{sec:borderline} and \ref{sec:supercritical} are devoted to the proofs of
Theorems~\ref{main}, \ref{anis-main-theorem} and \ref{pn-thm-main}, respectively.

\section{Preliminaries}\label{sec:prelim}

In this section we fix the notation and recall some preliminary results that will be
used throughout the paper: the basic properties of the Finsler norm and of the
logarithmic weights, two pointwise convexity inequalities, some Hardy-type inequalities
and the ground-state reduction of the Hardy difference.

\subsection{Notation and basic facts}\label{ssec:finsler}

Throughout the paper, $F:\Rn\to[0,\infty)$, $n\ge2$, is an even, positively
$1$-homogeneous and strongly convex norm, namely $F\in C^2(\Rn\setminus\{0\})$ and
$\nabla^2F^2$ is positive definite away from the origin. The polar norm of $F$ is
\[\Fo(x):=\sup_{F(\xi)\le1}x\cdot\xi;\] it is again an even, positively $1$-homogeneous,
strongly convex norm of class $C^2(\Rn\setminus\{0\})$, and $(\Fo)_\circ=F$ (see e.g.\
\cite{Schn}).

In the following we write $F_\xi:=\nabla F$, $\uw_\rho:=\{\Fo<\rho\}$, $\uw:=\uw_1$ for the
associated Wulff shape and $\omega_{n,F}:=\Ln(\uw)$ (see e.g.\ \cite{BFK}).
By homogeneity there exist positive constants such that
\begin{equation}\label{ellip}
 \lambda_F|\xi|\le F(\xi)\le\Lambda_F|\xi|,\qquad
 \lambda_\circ|x|\le\Fo(x)\le\Lambda_\circ|x|,\qquad |F_\xi|\le L_F,
\end{equation}
with $\lambda_\circ=1/\Lambda_F$ and $\Lambda_\circ=1/\lambda_F$. Moreover
\begin{equation}\label{struct}
 \nabla^2\!\big(\tfrac12F^2\big)(\eta)=F_\xi(\eta)\otimes F_\xi(\eta)+F(\eta)\,\nabla^2F(\eta),
 \qquad
 \big\langle\nabla^2\!\big(\tfrac12F^2\big)(\eta)\zeta,\zeta\big\rangle\ge\Theta_F|\zeta|^2
\end{equation}
for every $\eta\neq0$ and $\zeta\in\Rn$, for some $\Theta_F>0$.
We shall repeatedly use the following identities, valid for $\xi\in\Rn$ and
$x\neq0$ (see for example \cite{BP}):
\begin{equation}\label{Cauchy Schwarz}
 |\xi\cdot x|\le F(\xi)\,\Fo(x),\qquad \nabla\Fo(x)\cdot x=\Fo(x),\qquad
 \Div\Big\{\frac{x}{\Fo^{\,s}}\Big\}=\frac{n-s}{\Fo^{\,s}}\quad(s\in\R),
\end{equation}
\begin{equation}\label{BP-identity}
 F\big(\nabla\Fo(x)\big)=1,\qquad F_\xi\big(\nabla\Fo(x)\big)=\frac{x}{\Fo(x)}.
\end{equation}
In particular, we recall that  $\Div\{x/\Fo^{\,n}\}=0$ in $\Rn\setminus\{0\}$.

In the following we will use the anisotropic spherical coordinates $x=r\theta$, with  $r=\Fo(x)$ and
$\theta=x/\Fo(x)\in\prt\uw$. By the coarea formula applied to $\Fo$ and the
$0$-homogeneity of $\nabla\Fo$, for every $g\in L^1(\Rn)$
\begin{equation}\label{polar}
 \int_{\Rn}g\,\dd x=\int_0^\infty r^{n-1}\int_{\prt\uw}g(r\theta)\,\dd\mu(\theta)\,\dd r,
 \qquad \dd\mu:=\frac{\dd\Hnmo}{|\nabla\Fo|}\ \text{ on }\prt\uw,
\end{equation}
and $\mu(\prt\uw)=n\,\omega_{n,F}$.
If $u\in C^1$ is written as $u(x)=g(r,\theta)$, then $$\nabla u(x)\cdot\theta=\prt_rg(r,\theta),$$
by the $0$-homogeneity of $\theta$ and $\nabla\Fo\cdot x=\Fo$; in particular using
\eqref{BP-identity} and the evenness of $F$ if $u=\tilde u(\Fo)$ is $\Fo$-radially symmetric then
$F(\nabla u)=|\tilde u'(\Fo)|$.

For a function $v$ we will denote by
\[
 v_0(r):=\aver_{\prt\uw}v(r\theta)\,\dd\mu(\theta),\qquad v_\perp:=v-v_0,
\]
its $\Fo$-spherical mean and its mean-free part, respectively. Since $\prt_r v_0$ is the
$\dd\mu$-average of $\prt_rv=\nabla v\cdot\theta$ and $|\nabla v\cdot\theta|\le F(\nabla v)$ by
\eqref{Cauchy Schwarz}, Jensen's inequality gives
\begin{equation}\label{jensen}
 n\,\omega_{n,F}\,|v_0'(r)|^q\le\int_{\prt\uw}|\nabla v\cdot\theta|^q(r\theta)\,\dd\mu(\theta)
 \le\int_{\prt\uw}F^q(\nabla v)(r\theta)\,\dd\mu(\theta),
\end{equation}
for every $q\ge1$ and $r>0$. Finally, if $h:(0,\infty)\to[0,\infty]$ is nonincreasing and $A$ is a measurable set with $|A|=|\uw_\rho|$, the
layer-cake formula gives
\begin{equation}\label{rearr}
 \int_Ah(\Fo)\,\dd x\le\int_{\uw_\rho}h(\Fo)\,\dd x=n\,\omega_{n,F}\int_0^\rho h(t)\,t^{n-1}\,\dd t .
\end{equation}

The logarithmic weights $X_1$ and $X_2$ are increasing, $X_i(1)=1$, and
$\tfrac{\dd}{\dd t}X_1(t)=X_1^2(t)/t$. We shall use the following properties.

\begin{lemma}\label{lem-X}
\begin{itemize}
\item[(i)]  Let $\al_0>-1$ and $\beta,R>0$. For all $r\in(0,R]$, all
$c>1/(\al_0+1)$ and every $D\ge e^\eta R$, where
\[
 \eta:=\max\Big\{0,\frac{(\beta-\al_0-1)c+1}{(\al_0+1)c-1}\Big\},
\]
we have
\[
 \int_0^r t^{\al_0}X_1^{-\beta}(t/D)\,\dd t\le c\,r^{\al_0+1}X_1^{-\beta}(r/D).
\]
\item[(ii)] Let $\beta,\kappa,R>0$. If $\kappa\le1$ and $0<t\le D$, then
\[
 X_1^{-\beta}(\kappa t/D)\le\Big(1+\log\frac1\kappa\Big)^\beta X_1^{-\beta}(t/D),
 \qquad
 X_1^\beta(\kappa t/D)\le X_1^\beta(t/D).
\]
If $\kappa>1$, $0<t\le R$ and $D\ge\kappa^2R$, then
\[
 X_1^{-\beta}(\kappa t/D)\le X_1^{-\beta}(t/D),
 \qquad
 X_1^\beta(\kappa t/D)\le2^\beta X_1^\beta(t/D).
\]
\end{itemize}
\end{lemma}

\begin{proof}
Part (i) is proved in
\cite[Lemma~2.5]{Ps}.
Part (ii) follows from the identity $X_1^{-1}(\kappa t/D)=X_1^{-1}(t/D)+\log\frac1\kappa$ if
$\kappa\le1$, and, if $\kappa>1$, from
\[
 X_1(\kappa t/D)=\frac{X_1(t/D)}{1-\log\kappa\,X_1(t/D)},
 \qquad
 \log\kappa\,X_1(t/D)\le\frac{\log\kappa}{1+2\log\kappa}<\frac12 . \qedhere
\]
\end{proof}

The following lemma collects the pointwise convexity inequalities used in the ground-state
reduction.

\begin{lemma}\label{lem-convexity}
Let $\xi_1,\xi_2\in\Rn$ and
\[
 \mathcal D:=F^p(\xi_1+\xi_2)-F^p(\xi_1)-pF^{p-1}(\xi_1)F_\xi(\xi_1)\cdot\xi_2 ,
\]
where the last term is understood as zero when $\xi_1=0$.
\begin{itemize}
\item[(i)] If $p\ge2$, there are $\sigma_F\ge1$ ($\sigma_F=1$ if $F=|\cdot|$) and $c_F>0$ such
that
\begin{equation}\label{convex}
 \mathcal D\ge\frac{2^{1-p}}{\sigma_F^{\,p}}F^p(\xi_2),
 \qquad
 \mathcal D\ge c_F\,F^{p-2}(\xi_1)F^2(\xi_2);
\end{equation}
one may take
\[
 c_F=\min\big\{p\Theta_F2^{1-p}\Lambda_F^{-2},\,2^{3-2p}\sigma_F^{-p}\big\}.
\]
\item[(ii)] If $1<p<2$, then
\begin{equation}\label{subq}
 \mathcal D\ge c_*\,\frac{F^2(\xi_2)}{(F(\xi_1)+F(\xi_2))^{2-p}},
 \qquad c_*:=\min\Big\{1,\frac{p(p-1)\Theta_F}{2\Lambda_F^2}\Big\}.
\end{equation}
\end{itemize}
\end{lemma}

\begin{proof}
The first inequality in \eqref{convex} is
\cite[\S4.1, Lemma~2]{BlPPs}
. The other two
inequalities follow from Taylor's formula along the segment $[\xi_1,\xi_1+\xi_2]$, and
to this aim we first need a lower bound for the Hessian of $F^p$. On
$\Rn\setminus\{0\}$ we have $\nabla^2(F^p)=pF^{p-2}B$, where, by \eqref{struct},
\[
 B:=(p-1)F_\xi\otimes F_\xi+F\nabla^2F=\nabla^2\big(\tfrac12F^2\big)+(p-2)F_\xi\otimes F_\xi .
\]
If $p\ge2$, then $B\ge\Theta_F\,\mathrm{Id}$ directly by \eqref{struct}. If $1<p<2$,
the last term in the expression of $B$ is negative, and we compensate it with the
rank-one part of $\nabla^2(\frac12F^2)$: writing
$A:=\langle\nabla^2(\frac12F^2)\zeta,\zeta\rangle$ and $s:=(F_\xi\cdot\zeta)^2$, we have
$A\ge s$ by \eqref{struct} and the convexity of $F$, and $A\ge\Theta_F|\zeta|^2$, so that
\[
 \langle B\zeta,\zeta\rangle=A-(2-p)s\ge(p-1)A\ge(p-1)\Theta_F|\zeta|^2 .
\]

We now prove the second inequality in \eqref{convex}. We may assume
$\xi_1,\xi_2\neq0$, since otherwise the inequality is trivial (recall that $c_F\le1$),
and we distinguish two cases, according to whether or not $\xi_2$ is small with respect
to $\xi_1$. If $F(\xi_2)\le F(\xi_1)/2$, the segment
$[\xi_1,\xi_1+\xi_2]$ stays in $\{F\ge F(\xi_1)/2\}$ and Taylor's formula gives
\[
 \mathcal D\ge p\Theta_F|\xi_2|^2\Big(\frac{F(\xi_1)}2\Big)^{p-2}\int_0^1(1-s)\,\dd s
 \ge p\Theta_F2^{1-p}\Lambda_F^{-2}F^{p-2}(\xi_1)F^2(\xi_2).
\]
If instead $F(\xi_2)>F(\xi_1)/2$, Taylor's formula is of no use, since the segment
may approach the origin, but the first inequality in \eqref{convex} gives
\[
 \mathcal D\ge2^{1-p}\sigma_F^{-p}F^{p-2}(\xi_2)F^2(\xi_2)
 \ge2^{3-2p}\sigma_F^{-p}F^{p-2}(\xi_1)F^2(\xi_2).
\]

It remains to prove \eqref{subq}. If the segment $[\xi_1,\xi_1+\xi_2]$ avoids the
origin, Taylor's formula and $F(\xi_1+s\xi_2)\le M:=F(\xi_1)+F(\xi_2)$ give
\[
 \mathcal D\ge\frac p2(p-1)\Theta_F|\xi_2|^2M^{p-2},
\]
whence \eqref{subq} by \eqref{ellip}.
Otherwise $\xi_1$ and $\xi_2$ are antiparallel and the segment passes through the
origin, where $F^p$ is not twice differentiable; in this case we argue by approximation:
since $n\ge2$ we may perturb $\xi_1$ into
$\xi_1+\e v$, with $v$ not parallel to $\xi_2$, and let $\e\to0$, both sides of \eqref{subq} being
continuous ($F^p\in C^1(\Rn)$ for $p>1$). This concludes the proof.
\end{proof}

\subsection{Hardy-type inequalities}\label{ssec:hardy-trace}

Throughout, $\Om\subset\Rn$ is a bounded domain containing the origin and
$R_\Om:=\sup_\Om\Fo$. We start with a weighted trace-Hardy inequality.

\begin{lemma}\label{pn-lem-trace}
Let $V\subset\Rn$ be a bounded Lipschitz domain with exterior unit normal $\nu$, let
$q\ge1$, $s\ne n$, and let $v$ be Lipschitz, compactly supported and vanishing near the
origin. Then
\begin{equation}\label{pn-trace}
 \int_V\Fo^{q-s}F^q(\nabla v)\,\dd x
 +\Big(\frac{n-s}{q}\Big)\Big|\frac{n-s}{q}\Big|^{q-2}
 \int_{\partial V}\frac{|v|^q}{\Fo^s}x\cdot\nu\,\dd\Hnmo
 \ge\Big|\frac{n-s}{q}\Big|^q\int_V\frac{|v|^q}{\Fo^s}\,\dd x.
\end{equation}
\end{lemma}

\begin{proof}
For $q=1$ inequality \eqref{pn-trace} follows from the divergence theorem applied to the field
$|v|x\Fo^{-s}$, together with \eqref{Cauchy Schwarz}. For $q>1$ we apply the case $q=1$
to $|v|^q$ and then Young's inequality, as in \cite[Lemma~2.1]{Ps}.
\end{proof}

Choosing $V\supset\operatorname{supp}v$ and $q=s=p\neq n$, \eqref{pn-trace} gives
\eqref{anisoHardyintrod} for functions vanishing near the origin. We also need the following
one-dimensional inequalities.

\begin{lemma}\label{lem-one-dimensional}
Let $1<p<n$, $p^*=np/(n-p)$ and $B\ge1$. There exists $c=c(n,p)>0$, independent of $B$,
such that the following inequalities hold for every locally absolutely continuous
$h:(0,1]\to\R$ with $h(1)=0$:
\begin{itemize}
\item[(i)] if $1<p<2$,
\begin{equation}\label{1d}
 \int_0^1r^{p-1}|h'|^pX_1^{2-p}\left(\frac{r}{B}\right)\,\dd r
 \ge c\left(\int_0^1r^{-1}|h|^{p^*}X_1^{1+p^*/p}\left(\frac{r}{B}\right)\,\dd r\right)^{p/p^*};
\end{equation}
\item[(ii)] if $2\le p<n$,
\begin{equation}\label{1d-L2}
 \int_0^1r|h'|^2\,\dd r
 \ge c\left(\int_0^1r^{-1}|h|^{2p^*/p}X_1^{1+p^*/p}\left(\frac{r}{B}\right)\,\dd r\right)^{p/p^*}.
\end{equation}
\end{itemize}
Moreover, the following inequality, which involves neither $n$ nor $p$, holds true:
if $0<R\le D$
and $f$ is Lipschitz, compactly supported in $(0,R]$ and zero near $0$, then
\begin{equation}\label{1d-leray}
 \int_0^\infty|f'(t)|^2t\,\dd t\ge\frac14\int_0^\infty f^2(t)X_1^2\left(\frac{t}{D}\right)\,\frac{\dd t}{t}.
\end{equation}
\end{lemma}

\begin{proof}
We first rewrite \eqref{1d} and \eqref{1d-L2} on the half-line. Given $h$ as in the
statement, set $L:=1+\log B\ge1$ and $g(t):=h(e^{-t})$ for $t\ge0$. Then $g$ is locally
absolutely continuous on $[0,\infty)$, $g(0)=h(1)=0$ and $X_1(e^{-t}/B)=(L+t)^{-1}$.
Since $r=e^{-t}$ gives $\dd r=-r\,\dd t$ and $r|h'(r)|=|g'(t)|$, we obtain
\[
 \int_0^1r^{p-1}|h'|^pX_1^{2-p}\left(\frac{r}{B}\right)\,\dd r=\int_0^\infty|g'(t)|^p(L+t)^{p-2}\,\dd t,
 \qquad
 \int_0^1r|h'|^2\,\dd r=\int_0^\infty|g'(t)|^2\,\dd t,
\]
and, for every $Q>0$,
\[
 \int_0^1r^{-1}|h|^{Q}X_1^{1+p^*/p}\left(\frac{r}{B}\right)\,\dd r=\int_0^\infty|g(t)|^{Q}(L+t)^{-1-p^*/p}\,\dd t .
\]
Hence \eqref{1d} and \eqref{1d-L2} are both of the form
\begin{equation}\label{1d-halfline}
 \Big(\int_0^\infty|g|^{Q}w\,\dd t\Big)^{1/Q}\le C\Big(\int_0^\infty|g'|^{P}v\,\dd t\Big)^{1/P},
\end{equation}
with
\[
w(t):=(L+t)^{-1-p^*/p}
\]
and $P=p$, $Q=p^*$, $v(t)=(L+t)^{p-2}$ in the case of \eqref{1d}, and $P=2$, $Q=2p^*/p$,
$v\equiv1$ in the case of \eqref{1d-L2}; in both cases $1<P\le Q<\infty$, since
$p^*>p$, and $P/Q=p/p^*$.
By the two-weight criterion for Hardy inequalities with $P\le Q$
(see \cite[Theorem~1]{Brad} and \cite[\S1.3.2, Theorem~2]{Ma}),
\eqref{1d-halfline} holds for every locally absolutely continuous $g$ with $g(0)=0$,
with a constant $C\le P^{1/Q}(P')^{(P-1)/P}A$, where
\[
 A:=\sup_{s>0}\Big(\int_s^\infty w\,\dd t\Big)^{1/Q}
 \Big(\int_0^s v^{-\frac1{P-1}}\,\dd t\Big)^{\frac{P-1}P} .
\]
Since
\[
\int_s^\infty w\,\dd t=\frac p{p^*}(L+s)^{-p^*/p},
\]
in the case of \eqref{1d} we have
\[
 \int_0^s(L+t)^{\frac{2-p}{p-1}}\,\dd t\le(p-1)(L+s)^{\frac1{p-1}},
 \qquad\text{whence}\qquad
 A\le\Big(\frac p{p^*}\Big)^{1/p^*}(p-1)^{\frac{p-1}p},
\]
while in the case of \eqref{1d-L2} we get
\[
 A\le\Big(\frac p{p^*}\Big)^{\frac p{2p^*}}\sup_{s>0}\Big(\frac s{L+s}\Big)^{1/2}
 \le\Big(\frac p{p^*}\Big)^{\frac p{2p^*}} .
\]
 Both bounds are independent of $L$ and depend only on $n$ and $p$, and this
proves \eqref{1d} and \eqref{1d-L2}. We recall that \eqref{1d-L2} also appears in the
proof of \cite[Lemma~4.1]{GkPs} with $k=0$, where it is derived from
\cite[Lemma~7.1]{FT}.

To prove \eqref{1d-leray}, set $X(t):=X_1(t/D)$, which is well defined on $(0,R]$ since
$R\le D$, and $g:=X^{1/2}f$, which is Lipschitz and vanishes near $0$ and on
$[R,\infty)$. Since $X'(t)=X^2(t)/t$, we have $(X^{-1/2})'=-\frac12X^{1/2}/t$, and hence
\[
 f'=X^{-1/2}g'-\frac12\,X^{1/2}\,\frac gt,
 \qquad
 |f'|^2t=X^{-1}|g'|^2t-gg'+\frac14\,X\,\frac{g^2}t .
\]
The term $gg'=\frac12(g^2)'$ integrates to zero, since $g$ vanishes near $0$ and for
$t\ge R$, and the first term is nonnegative. Therefore
\[
 \int_0^\infty|f'(t)|^2t\,\dd t\ge\frac14\int_0^\infty X\,\frac{g^2}t\,\dd t
 =\frac14\int_0^\infty f^2(t)X_1^2(t/D)\,\frac{\dd t}t . \qedhere
\]
\end{proof}

The next weighted Hardy inequality will be used in the subquadratic range.

\begin{lemma}\label{lem:weighted-Hardy-transform}
Let $1<p<n$, $D\ge R_\Om$ and $X:=X_1(\Fo/D)$. For every $u\in\test(\Om)$, the function
$w=\Fo^{n/p-1}u$ satisfies
\begin{equation}\label{eq:weighted-Hardy-transform}
 \int_\Om\Fo^{-n}|w|^pX^2\,\dd x\le p^p\int_\Om\Fo^{p-n}F^p(\nabla w)X^{2-p}\,\dd x.
\end{equation}
\end{lemma}

\begin{proof}
By density we may assume that $u$ vanishes near the origin. Since
$\Div(xX/\Fo^{\,n})=\Fo^{-n}X^2$, an integration by parts, \eqref{Cauchy Schwarz} and
H\"older's inequality give
\begin{align*}
 \int_\Om\Fo^{-n}|w|^pX^2\,\dd x&\le p\int_\Om|w|^{p-1}F(\nabla w)\Fo^{1-n}X\,\dd x\\
 &\le p\Big(\int_\Om\Fo^{-n}|w|^pX^2\,\dd x\Big)^{\frac{p-1}p}
 \Big(\int_\Om\Fo^{p-n}F^p(\nabla w)X^{2-p}\,\dd x\Big)^{\frac1p}. \qedhere
\end{align*}
\end{proof}

\subsection{The ground-state reduction}\label{ssec:gs-reduction}

Fix $1<p\neq n$ and set $\beta:=(n-p)/p$ and $\phi:=\Fo^{-\beta}$. By \eqref{BP-identity}
and the $0$-homogeneity and oddness of $F_\xi$, we get
\begin{equation}\label{gradphi}
 \nabla\phi=-\beta\Fo^{-\beta-1}\nabla\Fo,\qquad
 F(\nabla\phi)=|\beta|\Fo^{-\beta-1},\qquad
 F_\xi(\nabla\phi)=-\sgn(\beta)\,\frac{x}{\Fo},
\end{equation}
and $\phi$ is the ground state of the anisotropic Hardy operator, that is, it solves
\[
 -\Div\big(F^{p-1}(\nabla\phi)F_\xi(\nabla\phi)\big)=|\beta|^p\,\frac{\phi^{p-1}}{\Fo^{\,p}}
 \qquad\text{in }\Rn\setminus\{0\}.
\]

Let $u\in\test(\Om)$ if $1<p<n$, and $u\in\test(\Om\setminus\{0\})$ if $p>n$, and set
$w:=\Fo^{\beta}u=u/\phi$. We split $\nabla u=\xi_1+\xi_2$, with $\xi_1:=w\nabla\phi$ and
$\xi_2:=\phi\nabla w$, so that, by \eqref{gradphi},  we have
\begin{equation}\label{key-eq}
 F^p(\xi_1)=\Big|\frac{n-p}{p}\Big|^{p}\frac{|u|^p}{\Fo^{\,p}},\quad
 F^p(\xi_2)=\Fo^{\,p-n}F^p(\nabla w),\quad
 F^{p-2}(\xi_1)F^2(\xi_2)=\frac{|\beta|^{p-2}|w|^{p-2}}{\Fo^{\,n-2}}F^2(\nabla w),
\end{equation}
where we used $\beta p=n-p$ and $(\beta+1)(p-2)+2\beta=n-2$.

The cross term in the expansion of $F^p(\xi_1+\xi_2)$ has zero integral, namely
\begin{equation}\label{cross}
 \int_\Om F^{p-1}(\xi_1)F_\xi(\xi_1)\cdot\xi_2\,\dd x=0 .
\end{equation}
Indeed, by \eqref{gradphi}, and since $(\beta+1)(p-1)+1+\beta=n$,
\[
 pF^{p-1}(\xi_1)F_\xi(\xi_1)\cdot\xi_2=-|\beta|^{p-2}\beta\,\frac{x\cdot\nabla(|w|^p)}{\Fo^{\,n}} .
\]
If $u$ vanishes near the origin, \eqref{cross} follows by an integration by parts, since
$\Div\{x/\Fo^{\,n}\}=0$. If $1<p<n$ and $u\in\test(\Om)$, we integrate by parts on
$\Om\setminus\overline{\uw_\e}$: since $|w|^p=\Fo^{\,n-p}|u|^p$, the integrand is $O(\Fo^{-p})$
and the boundary term on $\prt\uw_\e$ is $O(\e^{\,n-p})$, so that we may let $\e\to0$.

\begin{proposition}\label{prop-clean}
Let $u$ and $w=\Fo^{\beta}u$ be as above.
\begin{itemize}
\item[(i)] If $2\le p\neq n$, then
\begin{equation}\label{clean}
 I_{p,F}[u]\ge\frac{2^{1-p}}{\sigma_F^{\,p}}\int_\Om\Fo^{\,p-n}F^p(\nabla w)\,\dd x,
 \qquad
 I_{p,F}[u]\ge c_F|\beta|^{p-2}\int_\Om\Fo^{\,2-n}|w|^{p-2}F^2(\nabla w)\,\dd x .
\end{equation}
\item[(ii)] If $1<p<2$ and $D\ge R_\Om$, then
\begin{equation}\label{clean-subquadratic}
 I_{p,F}[u]\ge C_2\int_\Om\Fo^{\,p-n}F^p(\nabla w)\,X_1^{2-p}\Big(\frac{\Fo}{D}\Big)\,\dd x,
\end{equation}
\end{itemize}
where
\begin{equation}\label{C2}
C_2:=c_*\big[2^{p-1}\big(1+(n-p)^p\big)\big]^{-\frac{2-p}p}
\end{equation}
and $\sigma_F$, $c_F$ and $c_*$ are constants as in Lemma~\ref{lem-convexity}.
\end{proposition}

\begin{proof}
By \eqref{key-eq} and \eqref{cross}, the integral of $\mathcal D$ (see
Lemma~\ref{lem-convexity}) equals $I_{p,F}[u]$. Hence \emph{(i)} follows by integrating
\eqref{convex} and using \eqref{key-eq}.

As regards \emph{(ii)}, let $1<p<2$, $X:=X_1(\Fo/D)$ and
$S:=F(\xi_1)+F(\xi_2)$. By \eqref{subq},
\[I_{p,F}[u]\ge c_*M_2,
\]
where we set
\[
 M_2:=\int_\Om F^2(\xi_2)S^{p-2}\,\dd x,
 \qquad
 M_1:=\int_\Om F^p(\xi_2)X^{2-p}\,\dd x.
\]
 Since $M_1$ is, by
\eqref{key-eq}, the right-hand side of \eqref{clean-subquadratic} up to the constant
$C_2$, it is enough to bound $M_1$ by a multiple of $M_2$. By H\"older's inequality with
exponents $2/p$ and $2/(2-p)$ we have
\[
 M_1\le M_2^{p/2}\Big(\int_\Om S^pX^2\,\dd x\Big)^{(2-p)/2}.
\]
The last integral is in turn controlled by $M_1$ itself. Since $0<X\le1$,
$S^p\le2^{p-1}(F^p(\xi_1)+F^p(\xi_2))$ and $F^p(\xi_1)=\beta^p\Fo^{-n}|w|^p$,
Lemma~\ref{lem:weighted-Hardy-transform} yields
\[
 \int_\Om S^pX^2\,\dd x\le2^{p-1}\big(1+(\beta p)^p\big)M_1 .
\]
Inserting this bound in the previous inequality and simplifying, we get
$M_1\le[2^{p-1}(1+(n-p)^p)]^{(2-p)/p}M_2$, and \eqref{clean-subquadratic} follows.
\end{proof}

As a first consequence we obtain an $X_1^2$-improvement of the anisotropic Hardy
inequality, which will play a key role in the supercritical range, and whose sharpness
yields the optimality statements of Theorems~\ref{main} and \ref{pn-thm-main}.
The critical exponent in Theorem~\ref{anis-main-theorem} is addressed separately
in Section~\ref{sec:borderline}.

\begin{theorem}\label{pn-thm-X2}
Let $1<p<\infty$, $p\neq n$. Then for every $u\in\test(\Om\setminus\{0\})$ and every $D\ge R_\Om$,
\begin{equation}\label{pn-X2}
 I_{p,F}[u]\ \ge\ c_{n,p,F}
 \int_\Om\frac{|u|^p}{\Fo^{\,p}}\,X_1^2\Big(\frac{\Fo}{D}\Big)\,\dd x ,
 \qquad
 c_{n,p,F}:=\begin{cases}
 \dfrac{c_F\,|\beta|^{p-2}}{p^{2}} & \text{if } p\ge2,\\[6pt]
 \dfrac{C_2}{p^{p}} & \text{if } 1<p<2 ,
 \end{cases}
\end{equation}
with $C_2$ as in \eqref{C2} and $c_F$  as in Lemma~\ref{lem-convexity}. Moreover, the weight \(X_1^2\) is optimal, in the sense that the exponent \(2\)
cannot be decreased.
\end{theorem}

\begin{proof}
If $1<p<2$ (so that $p<n$), \eqref{pn-X2} follows at once from
Lemma~\ref{lem:weighted-Hardy-transform} and \eqref{clean-subquadratic}, since
$\Fo^{-n}|w|^p=\Fo^{-p}|u|^p$. Let now $p\ge2$ and
let $\psi:=|w|^{p/2}$, which is Lipschitz and compactly supported in $\Om\setminus\{0\}$.
Since
\[
 |\nabla\psi\cdot\theta|^2=\frac{p^2}{4}|w|^{p-2}|\nabla w\cdot\theta|^2
 \le\frac{p^2}4|w|^{p-2}F^2(\nabla w),
\]
using the second inequality in \eqref{clean}, \eqref{polar}
and \eqref{1d-leray} applied on  $\psi(s\theta)$, we obtain
\begin{align*}
 \frac{p^2\,I_{p,F}[u]}{4c_F|\beta|^{p-2}}
 &\ge\int_{\prt\uw}\int_0^\infty\big|\prt_s\psi(s\theta)\big|^2s\,\dd s\,\dd\mu\\
 &\ge\frac14\int_{\prt\uw}\int_0^\infty\psi^2(s\theta)X_1^2(s/D)\,\frac{\dd s}{s}\,\dd\mu
 =\frac14\int_\Om\frac{|w|^p}{\Fo^{\,n}}X_1^2\Big(\frac{\Fo}D\Big)\dd x,
\end{align*}
that is \eqref{pn-X2} since $|w|^p\Fo^{-n}=|u|^p\Fo^{-p}$.

Now we want to prove that the weight $X_1^2$ is optimal, that is, that for every
$D\ge R_\Om$ and every $\gamma<2$ there is no constant $C>0$ such that
\begin{equation}\label{X2-sharp}
 I_{p,F}[u]\ \ge\ C\int_\Om\frac{|u|^p}{\Fo^{\,p}}\,X_1^{\gamma}\Big(\frac{\Fo}{D}\Big)\,\dd x
 \qquad\text{for all }u\in\test(\Om\setminus\{0\}).
\end{equation}
By monotonicity ($X_1\le1$), it is enough to consider $0<\gamma<2$. Fix $\rho>0$
with $\overline{\uw_\rho}\subset\Om$ and test \eqref{X2-sharp} on $\Fo$-radial
functions $u=f(\Fo)$ with $f\in C^1_c((0,\rho))$. Since $F(\nabla u)=|f'(\Fo)|$,
\eqref{polar} gives
\begin{align*}
 I_{p,F}[u]
 &=n\omega_{n,F}\int_0^\rho
 \left(|f'|^p-|\beta|^p\frac{|f|^p}{r^p}\right)r^{n-1}\,\dd r,\\
 \int_\Om\frac{|u|^p}{\Fo^p}X_1^\gamma\!\left(\frac{\Fo}D\right)\,\dd x
 &=n\omega_{n,F}\int_0^\rho |f|^p r^{n-p-1}
 X_1^\gamma(r/D)\,\dd r .
\end{align*}
Since the factor $n\,\omega_{n,F}$ is common to both sides, \eqref{X2-sharp}
restricted to $\Fo$-radial functions is equivalent to its Euclidean counterpart
for radial functions in $C^1_c(B_\rho\setminus\{0\})$ and we can argue as in \cite[Theorem~5.1(ii), with $k=n$]{BFT1}. Since $D\ge R_\Om\ge\rho$, that proof gives a contradiction. This concludes the proof.

\end{proof}

In the borderline case $p=n$ the same scheme, with the Leray ground state
$X_1^{-(n-1)/n}(\Fo/R_\Om)$ in place of $\phi$, gives the following lower bounds, proved in
\cite[Proposition~2.1 and Remark~2.1]{BlPPsNA}.

\begin{proposition}\label{proposition:link}
Let $u\in\test(\Om\setminus\{0\})$, $X_i:=X_i(\Fo/R_\Om)$, $i=1,2$, $v:=X_1^{1-1/n}u$ and
$\kappa_n:=\frac2n\big(\frac n{n-1}\big)^{n-2}$. Then
\begin{align}
 \int_\Om\Fo^{2-n}|v|^{n-2}\Big|\nabla v\cdot\frac{x}{\Fo}\Big|^2X_1^{-1}\,\dd x
 &\le\kappa_n I_{n,F}[u],\label{link}\\
 \int_\Om F^n(\nabla v)X_1^{1-n}\,\dd x&\le2^{n-1}\sigma_F^nI_{n,F}[u].\label{link2}
\end{align}
\end{proposition}

Finally, we need the following radial estimate, which is the anisotropic counterpart of
\cite[Proposition~3.3]{BlPPsIUMJ} and is implicit in the proof of
\cite[Theorem~3.2]{BlPPsNA}; we include its short proof for completeness.

\begin{proposition}\label{proposition key}
Let $q\ge1$ and $X_i=X_i(\Fo)$, $i=1,2$. For every $\Fo$-radial function $f$, Lipschitz
on $\overline\uw$ and vanishing on $\prt\uw$,
\begin{multline*}
 \Big(\aver_\uw\big(|f|^{2/n}X_1^{(1-n)/n}X_2^{1/n}\big)^q\,\dd x\Big)^{1/q}\\
 \le\frac{e^{n/q}}{n}\,\Gamma\Big(1+q\,\frac{n-1}{n}\Big)^{1/q}
 \Big(\aver_\uw\Fo^{2-n}\Big|\nabla f\cdot\frac{x}{\Fo}\Big|^2X_1^{-1}\,\dd x\Big)^{1/n}.
\end{multline*}
\end{proposition}

\begin{proof}
Write $f=\tilde f(\Fo)$ and set
\[
 A:=\aver_\uw\Fo^{2-n}\Big|\nabla f\cdot\frac{x}{\Fo}\Big|^2X_1^{-1}\,\dd x .
\]
Since $|\nabla f\cdot x/\Fo|=|\tilde f'(\Fo)|$, $X_1^{-1}(t)=1-\log t$ and
$|\uw|=\omega_{n,F}$, \eqref{polar} gives
\[
 A=n\int_0^1t|\tilde f'(t)|^2(1-\log t)\,\dd t .
\]
As $\tilde f(1)=0$, the Cauchy--Schwarz inequality yields, for $r\in(0,1]$,
\[
 |\tilde f(r)|^2\le\int_r^1t|\tilde f'(t)|^2(1-\log t)\,\dd t\int_r^1\frac{\dd t}{t(1-\log t)}
 \le\frac An\,\log X_1^{-1}(r),
\]
since the first integral is at most $A/n$ and the second one equals
$\log(1-\log r)=\log X_1^{-1}(r)$.
Since $X_2=(1+\log X_1^{-1})^{-1}$, this gives $|f|^{2/n}X_2^{1/n}\le(A/n)^{1/n}$ in
$\uw$, and
\[
 \aver_\uw\big(|f|^{2/n}X_1^{(1-n)/n}X_2^{1/n}\big)^q\,\dd x
 \le\Big(\frac An\Big)^{q/n}n\int_0^1t^{n-1}(1-\log t)^k\,\dd t ,
\]
 with $k:=q(n-1)/n$. The substitution $t=e^{1-s}$ gives
\[
 n\int_0^1t^{n-1}(1-\log t)^k\,\dd t=n\,e^n\int_1^\infty e^{-ns}s^k\,\dd s\le e^n n^{-k}\,\Gamma(1+k).
\]
Taking the $q$-th root we obtain the claim.
\end{proof}

\section{The subcritical range: proof of Theorem~\ref{main}}\label{sec:subcritical}

In this section we prove Theorem~\ref{main}. Let $1<p<n$, $u\in\test(\Om)$ and, as in
\S\ref{ssec:gs-reduction}, $\beta=(n-p)/p$ and $w=\Fo^{\beta}u$. We recall the anisotropic
Sobolev inequality
\cite[\S3]{BlPPs}
\begin{equation}\label{aniso-sob}
 \Big(\int_{\Rn}|f|^{p^*}\,\dd x\Big)^{1/p^*}
 \le S_{n,p,F}\Big(\int_{\Rn}F^p(\nabla f)\,\dd x\Big)^{1/p},
 \qquad f\in\test(\Rn).
\end{equation}
Let $D=BR_\Om$ with $B\ge1$, and set $X:=X_1(\Fo/D)$, $a:=1+p^*/p$ and
\[
 \mathcal E[w]:=\int_\Om\Fo^{\,p-n}F^p(\nabla w)\,X^{(2-p)_+}\,\dd x .
\]
where $s_+:=\max\{s,0\}$; thus the weight $X^{2-p}$ appears only for $1<p<2$.
By Proposition~\ref{prop-clean},
\begin{equation}\label{I}
I_{p,F}[u]\ge c\,\mathcal E[w]
\end{equation}
with $c=c(n,p,F)>0$.
Since $\beta p^*=n$, we have $|u|^{p^*}=\Fo^{-n}|w|^{p^*}$, so that it is enough to prove
\begin{equation}\label{eq:reduced-Sobolev}
 \Big(\int_\Om\Fo^{-n}|w|^{p^*}X^{a}\,\dd x\Big)^{p/p^*}\le C(n,p,F)\,I_{p,F}[u].
\end{equation}
The proof of \eqref{eq:reduced-Sobolev} proceeds in three steps. We first reduce to
the case in which $\Om$ is the Wulff shape; then we split $w$ into its $\Fo$-spherical
mean $w_0$ and its mean-free part $w_\perp$, and we estimate the two parts separately:
the radial part through the one-dimensional inequalities of
Lemma~\ref{lem-one-dimensional}, the mean-free part through a Poincar\'e inequality on
$\prt\uw$ and the anisotropic Sobolev inequality \eqref{aniso-sob}. The optimality of
the exponent $a$ is proved at the end of the section.

We start with the reduction. Both sides of \eqref{eq:reduced-Sobolev} scale in the
same way under $x\mapsto R_\Om x$ and
$\Om\subset\uw_{R_\Om}$; hence, extending $u$ by zero, we may assume that $\Om=\uw$, $R_\Om=1$
and $D=B$. We split $w=w_0+w_\perp$ and, by Minkowski's inequality in
$L^{p^*}(\uw,\Fo^{-n}X^a\dd x)$, we estimate the two parts separately.

We start with the radial part. By \eqref{polar},
\begin{equation}\label{radial_1}
 \int_\uw\Fo^{-n}|w_0|^{p^*}X^a\,\dd x=n\,\omega_{n,F}\int_0^1r^{-1}|w_0|^{p^*}X_1^a(r/B)\,\dd r,
\end{equation}
and $w_0(1)=0$. If $1<p<2$, \eqref{polar} and \eqref{jensen} with $q=p$ give
\begin{equation}\label{radial_2}
 n\,\omega_{n,F}\int_0^1r^{p-1}|w_0'|^pX_1^{2-p}(r/B)\,\dd r\le\mathcal E[w],
\end{equation}
and \eqref{1d}, applied to $h=w_0$, gives

\begin{equation}\label{radial_3}
\int_0^1r^{p-1}|w_0'|^pX_1^{2-p}(r/B)\,\dd r\ge c\left(\int_0^1r^{-1}|w_0|^{p^*}X_1^a(r/B)\,\dd r\right)^{p/p^*}.
\end{equation}
Substituting \eqref{radial_3} in \eqref{radial_2}, using \eqref{radial_1} and \eqref{I} we get \eqref{eq:reduced-Sobolev}.

If $2\le p<n$ we set $\psi:=|w|^{p/2}$. Since
$|\nabla\psi\cdot\theta|^2\le\frac{p^2}4|w|^{p-2}F^2(\nabla w)$, the second inequality in
\eqref{clean}, together with \eqref{polar} and \eqref{jensen} with $q=2$, gives
\[
 I_{p,F}[u]\ge\frac{4c_F\beta^{p-2}}{p^2}\int_\uw\Fo^{2-n}|\nabla\psi\cdot\theta|^2\,\dd x
 \ge\frac{4c_F\beta^{p-2}n\,\omega_{n,F}}{p^2}\int_0^1r\,|\psi_0'(r)|^2\,\dd r .
\]
By Jensen's inequality $\psi_0\ge|w_0|^{p/2}$, so that \eqref{1d-L2} applied to $h=\psi_0$ and combined with the previous inequality,
gives
\[
 \Big(\int_0^1r^{-1}|w_0|^{p^*}X_1^a(r/B)\,\dd r\Big)^{p/p^*}\le C\,I_{p,F}[u],
\]
which is \eqref{eq:reduced-Sobolev}, by \eqref{I}.

We now turn to the mean-free part, which we treat as in the Euclidean case
\cite[\S4]{GkPs}, in two moves: a Poincar\'e inequality on $\prt\uw$ controls the zero-order
term $\Fo^{-n}|w_\perp|^p$ by the reduced energy, and the anisotropic Sobolev inequality
\eqref{aniso-sob}, applied to a suitable auxiliary function, then yields the critical
integrability.

By \eqref{jensen}, and since $F(\nabla w_0)=|w_0'|$, we
have
\[\mathcal E[w_0]\le\mathcal E[w], \quad \text{whence} \quad \mathcal E[w_\perp]\le2^p\mathcal E[w].
\]
For
fixed $r$, the function $w_\perp(r\,\cdot)$ has zero $\mu$-mean on $\prt\uw$, and its
tangential gradient is $r$ times the tangential part of $\nabla w_\perp(r\theta)$, whose
length is at most $\lambda_F^{-1}F(\nabla w_\perp)$ by \eqref{ellip}. Since $\prt\uw$ is a compact
connected hypersurface of class $C^2$ and $\mu$ is comparable to $\Hnmo$, the Poincar\'e
inequality on $(\prt\uw,\mu)$ (see e.g.\ \cite{H}) gives
\[
 \int_{\prt\uw}|w_\perp(r\theta)|^p\,\dd\mu(\theta)
 \le C(F,p)\,r^p\int_{\prt\uw}F^p(\nabla w_\perp)(r\theta)\,\dd\mu(\theta),
 \qquad r\in(0,1),
\]
and, multiplying by $r^{-1}X_1^{(2-p)_+}(r/B)$ and integrating,
\[
 \int_\uw\Fo^{-n}|w_\perp|^pX^{(2-p)_+}\,\dd x\le C\,\mathcal E[w_\perp]\le C\,\mathcal E[w].
\]
In order to apply \eqref{aniso-sob} we absorb the weights into the function. Let
$\gamma:=a/p^*=1/p+1/p^*$ and $z:=X^\gamma\Fo^{-\beta}w_\perp$, so that
$|z|^{p^*}=\Fo^{-n}|w_\perp|^{p^*}X^a$ is exactly the integrand we want to control.
We claim that
\[\int_\uw F^p(\nabla z)\,\dd x\le C\,\mathcal E[w].
\]
Indeed, since
$\nabla X=X^2\Fo^{-1}\nabla\Fo$,
\[
 \nabla z=X^{\gamma}\Fo^{-\beta}\nabla w_\perp
 +\big(\gamma X^{\gamma+1}-\beta X^{\gamma}\big)\Fo^{-\beta-1}\,w_\perp\nabla\Fo ,
\]
and, as $F(\nabla\Fo)=1$, $\beta p=n-p$, $\gamma p=2-p/n\ge(2-p)_+$ and $0<X\le1$,
\[
 \int_\uw F^p(\nabla z)\,\dd x\le C\int_\uw\Fo^{\,p-n}F^p(\nabla w_\perp)X^{(2-p)_+}\,\dd x
 +C\int_\uw\Fo^{-n}|w_\perp|^pX^{(2-p)_+}\,\dd x\le C\,\mathcal E[w].
\]
The function $z$ is Lipschitz on $\overline\uw\setminus\{0\}$, vanishes on $\prt\uw$, and
$z\to0$ at the origin, since $X(0)=0$ and $\Fo^{-\beta}|w_\perp|$ is bounded.

Since
$F(\nabla z)\in L^p$ and a point has zero $p$-capacity for $p<n$, $z\in W^{1,p}_0(\uw)$;
therefore \eqref{aniso-sob} applies to $z$ extended by zero and, together with \eqref{I}, gives
\[
 \Big(\int_\uw\Fo^{-n}|w_\perp|^{p^*}X^a\,\dd x\Big)^{1/p^*}
 \le C\,\big(\mathcal E[w]\big)^{1/p}
 \le C\,\big(I_{p,F}[u]\big)^{1/p}.
\]
This concludes the proof of \eqref{eq:reduced-Sobolev}.

Finally, we establish the optimality of the exponent $1+p^*/p$. The idea is to
interpolate, via H\"older's inequality, between a Hardy--Sobolev inequality with a
smaller exponent on $X_1$ and the Hardy weight $\Fo^{-p}$. This would produce a
remainder of the form $\Fo^{-p}X_1^{\vartheta}$ with $\vartheta<2$, which is forbidden by
Theorem~\ref{pn-thm-X2}.
Suppose, for the sake of contradiction, that $\e\in[0,1)$ is such that for some $D\geq R_\Om$ and some constant $C>0$
\[
  I_{p,F}[u]
 \ge
 C\Bigg(
 \int_\Om |u|^{p^\ast}
 X_1^{\,(1+p^\ast/p)\e}\!
 \Big(\frac{\Fo}{D}\Big)\,\dd x
 \Bigg)^{p/p^\ast}
 \qquad\text{for all }u\in\test(\Om).
\]
Let $\vartheta\in\R$. Applying H\"{o}lder's inequality with conjugate exponents $p^\ast/p$ and $n/p$ and using the hypothesis we get
\begin{align}\nonumber
 \int_{\Om}\frac{|u|^p}{\Fo^p}X_1^{\vartheta}\Big(\frac{\Fo}{D}\Big)~\dd x
   & = \int_{\Om}\Big\{|u|^pX_1^{(1+p/p^\ast)\e}\Big(\frac{\Fo}{D}\Big)\Big\}\Big\{\Fo^{-p}X_1^{\vartheta-(1+p/p^\ast)\e}\Big(\frac{\Fo}{D}\Big)\Big\}~\dd x
   \\ \label{cntrdctn} & \leq
    \frac1CI_{p,F}[u]\Bigg(\int_{\Om}\Fo^{-n}X_1^{\varrho}\Big(\frac{\Fo}{D}\Big)~\dd x\Bigg)^{p/n},
\end{align}
where $\varrho:=\big[\vartheta-\big(1+p/p^\ast\big)\e\big]\frac{n}{p}$. The integral on the right-hand side of \eqref{cntrdctn} does not depend on $u$, and it is finite if and only if $\varrho>1$, which is in turn equivalent to
\[
 \vartheta>2-\bigg(1+\frac{p}{p^\ast}\bigg)(1-\e).
\]
Since $\e<1$, values of $\vartheta$ less than $2$ are allowed in \eqref{cntrdctn}, and this contradicts the sharpness of the exponent $2$ on $X_1$ in \eqref{pn-X2}. Theorem~\ref{main} is proved.

\section{The borderline case: proof of Theorem~\ref{anis-main-theorem}}\label{sec:borderline}

 As recalled in the
Introduction, Theorem~\ref{anis-main-theorem} follows from the growth estimate
\eqref{sharp Leray-Trudinger growth} by the standard Taylor-series argument, so that we
focus our attention on \eqref{sharp Leray-Trudinger growth}, and we start with three
reductions. Since $\Om\subset\uw_{R_\Om}$ and $u$ is
extended by zero, the averages over $\Om$ in \eqref{sharp Leray-Trudinger growth} are
bounded by $|\uw_{R_\Om}|/|\Om|$ times the corresponding
averages over $\uw_{R_\Om}$ (to the power $1/q$),
while $I_{n,F}[u]$ is unchanged; by scaling, we may thus assume that $\Om=\uw$.
Moreover, \eqref{sharp Leray-Trudinger growth} is
positively homogeneous of degree one in $u$, so that it is enough to consider
functions with $I_{n,F}[u]\le1$. Finally, since a point has zero
$n$-capacity \cite[Theorem~2.43]{HKM}, we may assume that $u$ vanishes near the
origin. Indeed, if $u\in\test(\uw)$ and $u_k:=\eta_ku$, where $\eta_k$ are logarithmic
cut-off functions vanishing near the origin, then
\[
 \int_\uw F^n(\nabla u_k)\,\dd x\longrightarrow\int_\uw F^n(\nabla u)\,\dd x,
\]
the Leray terms converge by dominated convergence, since $\Fo^{-n}X_1^n\in L^1(\uw)$, and the
left-hand side of \eqref{sharp Leray-Trudinger growth} passes to the limit by Fatou's lemma.
Summing up, it is enough to prove \eqref{sharp Leray-Trudinger growth} for $\Om=\uw$
and $u\in\test(\uw\setminus\{0\})$ with $I_{n,F}[u]\le1$.

As in
Proposition~\ref{proposition:link} we set $X_i=X_i(\Fo/R_\Om)$. The strategy is the same as in Section~\ref{sec:subcritical}: we split $u$ into a
radial part and a remainder, and we estimate them separately; as we shall see, the
doubly logarithmic weight is carried entirely by the radial part. More precisely,
we set
\[
v:=X_1^{1-1/n}u,\quad \quad  u_0:=v_0X_1^{-1+1/n},
\]
so that
$u-u_0=v_\perp X_1^{-1+1/n}$, and for $q>n$ we write
\begin{equation}\label{Mink}
 \Big(\aver_{\uw}\big(|u|X_2^{1/n}\big)^q\,\dd x\Big)^{1/q}
 \le\Big(\aver_{\uw}\big(|u_0|X_2^{1/n}\big)^q\,\dd x\Big)^{1/q}
 +\Big(\aver_{\uw}\big(|u-u_0|X_2^{1/n}\big)^q\,\dd x\Big)^{1/q}.
\end{equation}

We start with the first term in \eqref{Mink}. Let $\psi:=|v|^{n/2}$; by Jensen's inequality
$|v_0|\le\psi_0^{2/n}$, so that
\[
 |u_0|X_2^{1/n}\le\psi_0^{2/n}X_1^{(1-n)/n}X_2^{1/n}.
\]
The function $\psi_0$ is $\Fo$-radial, Lipschitz on $\overline\uw$ and vanishes on $\prt\uw$, so that
Proposition~\ref{proposition key}
with $f=\psi_0$ and Stirling's formula give
\[
 \Big(\aver_{\uw}\big(|u_0|X_2^{1/n}\big)^q\,\dd x\Big)^{1/q}
 \le c(n)\,q^{1-1/n}\Big(\aver_{\uw}\Fo^{2-n}|\psi_0'|^2X_1^{-1}\,\dd x\Big)^{1/n}.
\]
Since the weight $\Fo^{2-n}X_1^{-1}$ is radial, \eqref{jensen} with $q=2$, the identity
$|\nabla\psi\cdot\theta|^2=\frac{n^2}4|v|^{n-2}|\nabla v\cdot\theta|^2$ and \eqref{link},
with $I_{n,F}[u]\le1$, give
\[
 \aver_{\uw}\Fo^{2-n}|\psi_0'|^2X_1^{-1}\,\dd x
 \le\aver_{\uw}\Fo^{2-n}|\nabla\psi\cdot\theta|^2X_1^{-1}\,\dd x
 \le\frac{n^2\kappa_n}{4\,\omega_{n,F}}\,I_{n,F}[u]\le\frac{n^2\kappa_n}{4\,\omega_{n,F}} .
\]
Hence the first term in \eqref{Mink} is bounded by $C(n,F)\,q^{1-1/n}$.

We now estimate the second term in \eqref{Mink}. Using $X_2\le1$ and the Euclidean
Trudinger growth estimate (see \cite{Tr} and \cite[Theorem~7.15]{GT}) applied to
$u-u_0\in W^{1,n}_0(\uw)$, give
\[
 \Big(\aver_{\uw}\big(|u-u_0|X_2^{1/n}\big)^q\,\dd x\Big)^{1/q}\leq c(n)\,q^{1-1/n}\,\big\|\nabla\big(v_\perp X_1^{-1+1/n}\big)\big\|_{L^n(\uw)}.
\]
Since
$F(\nabla X_1^{-1+1/n})=\frac{n-1}{n}X_1^{1/n}\Fo^{-1}$ and $|\xi|\le\lambda_F^{-1}F(\xi)$,
\[
 \int_\uw\big|\nabla(v_\perp X_1^{-1+1/n})\big|^n\,\dd x
 \le\frac{2^{n-1}}{\lambda_F^n}\int_\uw\Big(F^n(\nabla v_\perp)X_1^{1-n}
 +\Big(\frac{n-1}{n}\Big)^n\frac{|v_\perp|^n}{\Fo^{\,n}}X_1\Big)\,\dd x .
\]
By \eqref{link2}, it is enough to bound each of the two terms on the right-hand
side by
\[
 C(n,F)\int_\uw F^n(\nabla v)X_1^{1-n}\,\dd x,
\]
and this is done exactly as for the
mean-free part in Section~\ref{sec:subcritical}, with $p=n$, using a Jensen argument for
the gradient term and a Poincar\'e inequality on $\prt\uw$ for the zero-order term.
Indeed, since $F$ is even and subadditive and $F(\nabla v_0)=|v_0'|$, we have
$F(\nabla v_\perp)\le F(\nabla v)+|v_0'|$; as the weight $X_1^{1-n}$ is radial,
\eqref{jensen} with $q=n$ gives
$$\int_\uw|v_0'|^nX_1^{1-n}\,\dd x\le\int_\uw F^n(\nabla v)X_1^{1-n}\,\dd x,$$ and hence
\begin{equation}\label{gapI}
 \int_\uw F^n(\nabla v_\perp)X_1^{1-n}\,\dd x\le2^n\int_\uw F^n(\nabla v)X_1^{1-n}\,\dd x .
\end{equation}
Moreover, the Poincar\'e inequality on $(\prt\uw,\mu)$ applied to the mean-free slices
$v_\perp(r\,\cdot)$, whose tangential gradient is bounded by $r\lambda_F^{-1}F(\nabla v_\perp)$,
gives
\[
 \int_{\prt\uw}|v_\perp(r\theta)|^n\,\dd\mu(\theta)
 \le C\,r^n\int_{\prt\uw}F^n(\nabla v_\perp)(r\theta)\,\dd\mu(\theta);
\]
multiplying by $r^{-1}X_1(r)\le r^{-1}X_1^{1-n}(r)$, integrating and using \eqref{gapI}, we
obtain
\begin{equation}\label{gapII}
 \int_\uw\frac{|v_\perp|^n}{\Fo^{\,n}}X_1\,\dd x\le C(n,F)\int_\uw F^n(\nabla v)X_1^{1-n}\,\dd x .
\end{equation}
Hence also the second term in \eqref{Mink} is bounded by $C(n,F)\,q^{1-1/n}$, and
\eqref{sharp Leray-Trudinger growth} is proved.

Finally, we prove that the exponent $1/n$ on $X_2$ cannot be decreased; by
monotonicity it is enough to consider $0\le\gamma<1/n$, and to test with
$\Fo$-radial functions. Fix $\rho>0$ with $\overline{\uw_\rho}\subset\Om$ and let
$u=f(\Fo)$ with $f\in C_c^1((0,\rho))$. All the quantities in
Theorem~\ref{anis-main-theorem} reduce, up to the factor $n\,\omega_{n,F}$, to
one-dimensional ones which do not depend on $F$, therefore we can argue as in \cite[\S4, (4.1)--(4.4)]{MT} and \cite[Proposition~3.2, $k=n$, $m=2$]{BFTseries} reaching a contradiction.

\section{The supercritical range: proof of Theorem~\ref{pn-thm-main}}\label{sec:supercritical}

In this section we prove Theorem~\ref{pn-thm-main}. The proof follows the strategy of
the Euclidean case \cite{Ps}, and we only highlight the points where the anisotropy
enters. Let $p>n$ and $\al:=1-n/p\in(0,1)$, so that $\beta=-\al$, the ground state is
$\phi=\Fo^{\,\al}$ and, following \cite{Ps}, we denote the transformed function by
$v:=\Fo^{-\al}u$. For $u\in\test(\Om\setminus\{0\})$, $v$ is Lipschitz, compactly supported
in $\Om$ and vanishes near the origin. Proposition~\ref{prop-clean} and
Theorem~\ref{pn-thm-X2} give
\begin{align}
 I_{p,F}[u]&\ge\frac{2^{1-p}}{\sigma_F^{\,p}}\int_\Om\Fo^{\,p-n}F^p(\nabla v)\,\dd x,
 \label{pn-B1}\\
 I_{p,F}[u]&\ge c_F\,\al^{p-2}\int_\Om\Fo^{\,2-n}|v|^{p-2}F^2(\nabla v)\,\dd x,
 \label{pn-B2bis}\\
 I_{p,F}[u]&\ge\frac{c_F\,\al^{p-2}}{p^2}\int_\Om\frac{|v|^p}{\Fo^{\,n}}X_1^2\Big(\frac\Fo D\Big)\,\dd x,
 \qquad D\ge R_\Om,
 \label{pn-X2bis}
\end{align}
and, since $\nabla u=\Fo^{\,\al}\nabla v+\al\Fo^{\,\al-1}v\,\nabla\Fo$ and $F(\nabla\Fo)=1$,
\begin{equation}\label{pn-F-split}
 F(\nabla u)\le\Fo^{\,\al}F(\nabla v)+\al\,\Fo^{-n/p}\,|v| .
\end{equation}
Throughout this section $u$ and $v$ are extended by zero to $\Rn$, and
$\diam(\Om)$ is the Euclidean diameter of $\Om$.
 The symbols
$c,C,\eta$ denote positive constants and $\Theta$ a nonnegative constant, depending only
on $n,p,q,F$, whose value may change from line to line.

The Morrey argument relies on the following anisotropic representation formula.

\begin{lemma}\label{pn-lem-rep}
Let $B_r$ be a ball, let $u$ be Lipschitz on $\overline{B_r}$ and
\[
 u_{B_r}:=\frac1{|B_r|}\int_{B_r}u\,\dd z .
\]
Then, for every $x\in B_r$,
\begin{equation}\label{pn-rep-ball}
 |u(x)-u_{B_r}|\le\frac{2^n}{n\,\omega_n}\int_{B_r}F\big(\nabla u(z)\big)\,
 \frac{\Fo(z-x)}{|z-x|^{\,n}}\,\dd z
 \le\frac{2^n\Lambda_\circ}{n\,\omega_n}\int_{B_r}\frac{F(\nabla u(z))}{|z-x|^{\,n-1}}\,\dd z .
\end{equation}
\end{lemma}

\begin{proof}
For $x,z\in B_r$ and $e=(z-x)/|z-x|$ we have, by \eqref{Cauchy Schwarz},
\[
 |u(x)-u(z)|\le\int_0^{|z-x|}F\big(\nabla u(x+te)\big)\,\Fo(e)\,\dd t .
\]
Averaging over $z\in B_r$, using Euclidean polar coordinates centered at $x$ and
$|z-x|\le2r$, we get
\[
 |u(x)-u_{B_r}|\le\frac{(2r)^n}{n|B_r|}\int_{\mathbb S^{n-1}}\int_0^\infty
 V(x+te)\,\Fo(e)\,\dd t\,\dd e,
\]
where $V:=F(\nabla u)\chi_{B_r}$; undoing the
polar coordinates and using $\Fo(e)=\Fo(y-x)/|y-x|$ gives the first inequality, and
the second one follows from \eqref{ellip}.
\end{proof}

Let $g\in L^p(B_r)$. Since $B_r\subset B_{2r}(x)$ for $x\in B_r$ and $(1-n)p'+n=\frac{p-n}{p-1}>0$, H\"older's
inequality gives, for every $x\in B_r$,
\begin{equation}\label{pn-kernel}
 \int_{B_r}\frac{g(z)}{|z-x|^{\,n-1}}\,\dd z\le C(n,p)\,r^{\,\al}\,\|g\|_{L^p(B_r)},
\end{equation}
and in particular the anisotropic Morrey inequality
\[
 |u(x)-u(y)|\le C\,\Fo(x-y)^{\al}\,\|F(\nabla u)\|_{L^p(\Rn)},
 \qquad u\in\test(\Rn).
\]
The crucial step
is the following local improvement, which is the anisotropic version of
\cite[Theorem~C]{Ps}.

\begin{theorem}\label{pn-thm-local}
Let $p>n$ and $1\le q<p$. There exist $\Theta\ge0$ and $C>0$ such that, for all
$u\in\test(\Om\setminus\{0\})$, every ball $B_r$ with $r\in(0,\diam(\Om))$ and every
$D\ge e^{\Theta}\diam(\Om)$,
\begin{equation}\label{pn-local}
 r^{\,n/p}\,X_1^{1/p}(r/D)\,
 \bigg(\frac{1}{|B_r|}\int_{B_r}\frac{|u|^q}{\Fo^{\,q}}\,\dd x\bigg)^{\!1/q}
 \le C\,\big(I_{p,F}[u]\big)^{1/p}.
\end{equation}
\end{theorem}

\noindent The proof is divided into two steps. We first prove \eqref{pn-local} for balls
whose closure contains the origin, where Lemma~\ref{pn-lem-trace} can be applied
on $B_r$; this yields the pointwise bound of Proposition~\ref{pn-prop-origin},
which in turn allows us to treat arbitrary balls.

\begin{proof}[Proof of Theorem \ref{pn-thm-local}. Step 1]
Let $B_r$ be a ball with $0\in\overline{B_r}$ and $r\in(0,\diam(\Om))$, and set
$s:=nq/p\in(0,n)$. Since $\al-1=-n/p$, we have $|u|^q\Fo^{-q}=|v|^q\Fo^{-s}$, and
$(n-s)/q=\frac np\cdot\frac{p-q}q>0$. Moreover, if $c$ is the
center of $B_r$, then $x\cdot\nu=r+c\cdot\nu\ge r-|c|\ge0$ on $\prt B_r$. Hence Lemma~\ref{pn-lem-trace}
with $V=B_r$ gives
\begin{equation}\label{pn-loc-start}
 \Big(\frac{n-s}{q}\Big)^{\!q}\int_{B_r}\frac{|u|^q}{\Fo^{\,q}}\,\dd x
 \le N_r+\Big(\frac{n-s}{q}\Big)^{\!q-1}P_r,
\end{equation}
where
\[
 N_r:=\int_{B_r}\Fo^{\,q-s}F^q(\nabla v)\,\dd x,\qquad
 P_r:=\int_{\prt B_r}\frac{|v|^q}{\Fo^{\,s}}\,x\cdot\nu\,\dd\Hnmo .
\]
We first estimate $N_r$. Since $(q-s)p/q=p-n$, H\"older's inequality with exponents
$p/(p-q)$ and $p/q$ and \eqref{pn-B1} give, for every $D\ge\diam(\Om)$,
\begin{equation}\label{pn-Nr}
 N_r\le|B_r|^{\frac{p-q}p}\Big(\int_{B_r}\Fo^{\,p-n}F^p(\nabla v)\,\dd x\Big)^{\frac qp}
 \le C\,r^{\frac{n(p-q)}p}\big(I_{p,F}[u]\big)^{\frac qp}
 \le C\,r^{\frac{n(p-q)}p}X_1^{-\frac qp}(r/D)\,\big(I_{p,F}[u]\big)^{\frac qp},
\end{equation}
since $0<X_1\le1$. We now turn to $P_r$. Let $X:=X_1(\Fo/D)$ and $\gamma:=q/(p-q)$. Writing
the integrand of $P_r$ as $X^{-q/p}\cdot|v|^q\Fo^{-s}X^{q/p}$ and using H\"older's inequality
with exponents $p/(p-q)$ and $p/q$ with respect to the measure $x\cdot\nu\,\dd\Hnmo\ge0$,
together with $sp/q=n$, we get
\[
P_r\le S_r^{(p-q)/p}\,T_r^{q/p},
\] where
\[
 S_r:=\int_{\prt B_r}X^{-\gamma}\,x\cdot\nu\,\dd\Hnmo,\qquad
 T_r:=\int_{\prt B_r}\frac{|v|^p}{\Fo^{\,n}}\,X\,x\cdot\nu\,\dd\Hnmo .
\]
To estimate $S_r$ we observe that, by $\frac{\dd}{\dd t}X_1=X_1^2/t$ and Euler's identity,
\[
 \nabla X^{-\gamma}\cdot x=-\gamma X^{-\gamma-1}\,\frac{X^2}{\Fo/D}\,\frac{\nabla\Fo\cdot x}{D}
 =-\gamma X^{1-\gamma},
\]
so that
\[
 \Div\{X^{-\gamma}x\}=nX^{-\gamma}-\gamma X^{1-\gamma}\le nX^{-\gamma}.
\]
The field $X^{-\gamma}x$ is continuous up to the origin, with value $0$ there, and its
divergence is integrable. Moreover $\Fo\le2\Lambda_\circ\diam(\Om)\le D$ on $B_r$, since
$0\in\overline{B_r}$, and $t\mapsto X_1^{-\gamma}(t/D)$ is decreasing. Hence using the divergence
theorem, \eqref{rearr} with $|\uw_{\kappa_Fr}|=|B_r|$ and $\kappa_F:=\Big(\frac{\omega_n}{\omega_{n,F}}\Big)^{1/n}$, Lemma~\ref{lem-X}\emph{(i)} with $\al_0=n-1$ and
$\beta=\gamma$, and Lemma~\ref{lem-X}\emph{(ii)} with $\kappa=\kappa_F$, we obtain
\begin{equation}\label{pn-Sr}
 S_r\le n\int_{B_r}X^{-\gamma}\,\dd x\le n^2\omega_{n,F}\int_0^{\kappa_Fr}t^{n-1}X_1^{-\gamma}(t/D)\,\dd t
 \le C\,r^nX_1^{-\gamma}(r/D),
\end{equation}
for every $D\ge e^\eta\diam(\Om)$.

Finally, we estimate $T_r$. Since $v$ vanishes near the
origin, the field $|v|^p\Fo^{-n}X\,x$ is Lipschitz on $\overline{B_r}$; since
$\Div\{x/\Fo^{\,n}\}=0$ and $\nabla X\cdot x=X^2$, the divergence theorem gives
\[
 T_r=\int_{B_r}\frac{|v|^p}{\Fo^{\,n}}\,X^2\,\dd x
 +\int_{B_r}\frac{X}{\Fo^{\,n}}\;x\cdot\nabla\big(|v|^p\big)\,\dd x .
\]
By \eqref{Cauchy Schwarz}, $|x\cdot\nabla(|v|^p)|\le p|v|^{p-1}F(\nabla v)\Fo$ a.e., so that
\begin{equation*}
\begin{split}
 T_r\le& \int_\Om\frac{|v|^p}{\Fo^{\,n}}X^2\,\dd x
 +p\int_{B_r}\frac{X}{\Fo^{\,n}} |v|^{p-1}F(\nabla v)\Fo\,\dd x \\
 &=\int_\Om\frac{|v|^p}{\Fo^{\,n}}X^2\,\dd x
 +p\int_{B_r}\Big[\Fo^{\frac{2-n}2}|v|^{\frac{p-2}2}F(\nabla v)\Big]\Big[|v|^{\frac p2}\Fo^{-\frac n2}X\Big]\,\dd x\\
 &\leq \int_\Om\frac{|v|^p}{\Fo^{\,n}}X^2\,\dd x
 +p\Big(\int_\Om\Fo^{\,2-n}|v|^{p-2}F^2(\nabla v)\,\dd x\Big)^{\frac12}
 \Big(\int_\Om\frac{|v|^p}{\Fo^{\,n}}X^2\,\dd x\Big)^{\frac12}.
\end{split}
\end{equation*}
Therefore \eqref{pn-B2bis} and \eqref{pn-X2bis} give
\begin{equation}\label{pn-Tr}
 T_r\le\frac{2p^2}{c_F\,\al^{p-2}}\,I_{p,F}[u],\qquad D\ge R_\Om .
\end{equation}
Combining \eqref{pn-Sr} and \eqref{pn-Tr}, we get
\[
 P_r\le C\,r^{\frac{n(p-q)}p}X_1^{-\frac qp}(r/D)\,\big(I_{p,F}[u]\big)^{\frac qp},
\]
which, together with \eqref{pn-Nr}
and \eqref{pn-loc-start}, yields
\[
 \int_{B_r}\frac{|u|^q}{\Fo^{\,q}}\,\dd x\le C\,r^{\frac{n(p-q)}p}\,X_1^{-\frac qp}(r/D)\,
 \big(I_{p,F}[u]\big)^{\frac qp},\qquad D\ge e^{\Theta}\diam(\Om).
\]
Dividing by $|B_r|=\omega_nr^n$ and raising to the power $1/q$ we obtain \eqref{pn-local}.
\end{proof}

\begin{proposition}\label{pn-prop-origin}
There exist $\Theta\ge0$ and $C>0$ such that, for all $u\in\test(\Om\setminus\{0\})$ and
$D\ge e^{\Theta}\diam_F(\Om)$,
\begin{equation}\label{pn-origin}
 \sup_{x\in\Om}\bigg\{\frac{|u(x)|}{\Fo(x)^{\,\al}}\,
 X_1^{1/p}\Big(\frac{\Fo(x)}{D}\Big)\bigg\}\le C\,\big(I_{p,F}[u]\big)^{1/p} .
\end{equation}
\end{proposition}

\begin{proof}
The idea is to estimate the oscillation of $u$ on a ball containing the origin by means
of the representation formula of Lemma~\ref{pn-lem-rep},
since $u$
vanishes near the origin, the oscillation controls $u$ itself.
Let $B_r$ be a ball containing the origin, with $r\in(0,\diam(\Om))$, and let $x\in B_r$.
By \eqref{pn-rep-ball} and \eqref{pn-F-split}, and since $|v|\Fo^{-n/p}=|u|/\Fo$,
\begin{equation}\label{pn-origin-split}
 \frac{n\omega_n}{2^n\Lambda_\circ}\,|u(x)-u_{B_r}|\le K_r(x)+\al L_r(x),
\end{equation}
where
\[
 K_r(x):=\int_{B_r}\frac{\Fo^{\,\al}F(\nabla v)}{|x-z|^{n-1}}\,\dd z,\qquad
 L_r(x):=\int_{B_r}\frac{|u|/\Fo}{|x-z|^{n-1}}\,\dd z.
\]
Since $\Fo^{\al p}=\Fo^{\,p-n}$, \eqref{pn-kernel} and \eqref{pn-B1} give
\[
 K_r(x)\le C\,r^{\al}\big(I_{p,F}[u]\big)^{1/p}\le C\,r^{\al}X_1^{-1/p}(r/D)\big(I_{p,F}[u]\big)^{1/p}.
\]
As regards $L_r(x)$,
let $\e:=(p-n)/(2n)$ and $q:=p/(1+\e)\in[1,p)$. H\"older's inequality with exponents
$p/(p-1-\e)$ and $q$ gives
\[
 L_r(x)\le\left(\int_{B_{2r}(x)}\frac{\dd z}{|x-z|^{\frac{(n-1)p}{p-1-\e}}}\right)^{\frac{p-1-\e}p}
 \left(\int_{B_r}\left(\frac{|u|}{\Fo}\right)^{q}\dd z\right)^{\frac1q}.
\]
The kernel integral is finite, since $\e<(p-n)/n$, and equals
$C\,r^{(p-n-n\e)/(p-1-\e)}$, so that the first factor is $C\,r^{\al-n\e/p}$. By Step 1 of
the proof of Theorem~\ref{pn-thm-local} (recall that $0\in B_r$), and since $n/q-n/p=n\e/p$, the second factor is
bounded by $C\,r^{n\e/p}X_1^{-1/p}(r/D)\big(I_{p,F}[u]\big)^{1/p}$. Therefore
\begin{equation}\label{pn-osc}
 |u(x)-u_{B_r}|\le C\,r^{\al}X_1^{-1/p}(r/D)\,\big(I_{p,F}[u]\big)^{1/p}\qquad\text{for all }x\in B_r .
\end{equation}
Since $u$ vanishes in a neighbourhood of $0\in B_r$, evaluating \eqref{pn-osc} at points
close to the origin gives the same bound for $|u_{B_r}|$ and hence, by the triangle
inequality, for $|u(x)|$, $x\in B_r$. Given $x\in\Om$, we choose $B_r=B_{|x|}(x/2)$, which
contains $0$ and $x$ and has radius $|x|<\diam(\Om)$, and
we get
\[
 |u(x)|\le C\,|x|^{\al}X_1^{-1/p}(|x|/D)\big(I_{p,F}[u]\big)^{1/p}.
\]
Finally, since
$\lambda_\circ|x|\le\Fo(x)\le\Lambda_\circ|x|$ and $\diam(\Om)\le\lambda_\circ^{-1}\diam_F(\Om)$,
Lemma~\ref{lem-X}\emph{(ii)} converts $|x|$ into $\Fo(x)$ in both factors, at the cost of
enlarging $D$. This concludes the proof.
\end{proof}

\begin{proof}[Proof of Theorem \ref{pn-thm-local}. Step 2]
Let now $B_r$ be an arbitrary ball, with $r\in(0,\diam(\Om))$. By
Proposition~\ref{pn-prop-origin}, it follows that
\[
 \int_{B_r}\frac{|u|^q}{\Fo^{\,q}}\,\dd x\le C\,\big(I_{p,F}[u]\big)^{\frac qp}\int_{B_r\cap\Om}
 \Fo^{-\frac{nq}p}X_1^{-\frac qp}\Big(\frac\Fo D\Big)\,\dd x .
\]
To apply \eqref{rearr} without evaluating $X_1$ outside its domain, set
\[
 h_D(t):=\begin{cases}
 t^{-nq/p}X_1^{-q/p}(t/D),&0<t\le R_\Om,\\
 0,&t>R_\Om.
 \end{cases}
\]
This function is nonincreasing. Apply \eqref{rearr} to $B_r\cap\Om$, whose volume
is at most $|B_r|$, and then enlarge the comparison Wulff ball to $\uw_{\kappa_Fr}$.
Since $D$ may be chosen so that $\kappa_Fr\le D$, the untruncated weight is
well defined throughout this comparison ball. Thus \eqref{rearr},
Lemma~\ref{lem-X}\emph{(i)} with $\al_0=n-1-nq/p>-1$ and $\beta=q/p$, and Lemma~\ref{lem-X}\emph{(ii)} give
\[
 \int_{B_r\cap\Om}\Fo^{-\frac{nq}p}X_1^{-\frac qp}\Big(\frac\Fo D\Big)\,\dd x
 \le n\,\omega_{n,F}\int_0^{\kappa_Fr}t^{\,n-1-\frac{nq}p}X_1^{-\frac qp}(t/D)\,\dd t
 \le C\,r^{\,n-\frac{nq}p}X_1^{-\frac qp}(r/D),
\]
for $\kappa_Fr\le D$ and $D\ge e^\eta\diam(\Om)$, and \eqref{pn-local} follows as before.
\end{proof}

Now we are in position to prove Theorem~\ref{pn-thm-main}.
\begin{proof}[Proof of Theorem~\ref{pn-thm-main}]
Let $x,y\in\Om$, $x\neq y$, set $r:=|x-y|<\diam(\Om)$ and let $B_r$ be the ball of radius $r$
centered at $(x+y)/2$, so that $x,y\in B_r$. Then
\begin{equation}\label{pn-two-osc}
 |u(x)-u(y)|\le|u(x)-u_{B_r}|+|u(y)-u_{B_r}|,
\end{equation}
and it is enough to bound $|u(x)-u_{B_r}|$ by a quantity which does not depend on the
position of $x$ in $B_r$. As in \eqref{pn-origin-split},
$|u(x)-u_{B_r}|\le C(K_r(x)+L_r(x))$, and $K_r(x)\le Cr^{\al}\big(I_{p,F}[u]\big)^{1/p}$ as before.
This time $B_r$ need not contain the origin, and we estimate $L_r$ directly through
\eqref{pn-origin}, which gives
\[
 \frac{|u(z)|}{\Fo(z)}\le C\,\big(I_{p,F}[u]\big)^{1/p}\,\Fo(z)^{-n/p}\,X_1^{-1/p}\Big(\frac{\Fo(z)}D\Big).
\]
Let $Q:=(n+p)/2\in(n,p)$ and $Q':=Q/(Q-1)$. By H\"older's inequality,
\begin{equation*}
 L_r(x)\le C\,\big(I_{p,F}[u]\big)^{1/p}
 \Big(\int_{B_r\cap\Om}\Fo^{-\frac{nQ}p}X_1^{-\frac Qp}\Big(\frac\Fo D\Big)\,\dd z\Big)^{\frac1Q}
 \times\Big(\int_{B_{2r}(x)}\frac{\dd z}{|x-z|^{(n-1)Q'}}\Big)^{\frac1{Q'}}.
\end{equation*}
Both integrals are finite, since $nQ/p<n$ and $(n-1)Q'<n$. As in the proof of  Theorem~\ref{pn-thm-local}, now with $\al_0=n-1-nQ/p>-1$ and $\beta=Q/p$, the first
factor is bounded by $C\,r^{\frac nQ-\frac np}X_1^{-1/p}(r/D)$, while the second one equals
$C\,r^{\frac n{Q'}-n+1}$. Since $\frac nQ-\frac np+\frac n{Q'}-n+1=1-\frac np=\al$, we obtain
\[
 L_r(x)\le C\,r^{\al}X_1^{-1/p}(r/D)\big(I_{p,F}[u]\big)^{1/p}.
\]
The same bound holds at $y$, and
\eqref{pn-two-osc} gives
\[
 |u(x)-u(y)|\le C\,|x-y|^{\al}X_1^{-1/p}\big(|x-y|/D\big)\,\big(I_{p,F}[u]\big)^{1/p}.
\]
Converting $|x-y|$ into $\Fo(x-y)$ through \eqref{ellip} and Lemma~\ref{lem-X}\emph{(ii)}, both in
the power $\al$ and inside $X_1^{-1/p}$,
we obtain the desired estimate with the logarithmic scale
$D=BL$, where $L:=\diam_F(\Om)$ and $B=e^\Theta\ge1$.
For $0<t\le L$,
\[
 \frac{X_1(t/L)}{X_1(t/(BL))}
 =1+\frac{\log B}{1+\log(L/t)}\le1+\log B.
\]
Absorbing $(1+\log B)^{1/p}$ into $C(n,p,F)$ gives exactly \eqref{pn-HM}.

It remains to show that the exponent $1/p$ cannot be decreased. Assume on the contrary that there exists $\e\in(0,1]$ such that for all $u\in\test(\Om\setminus\{0\})$ we have
\[
 \sup_{\substack{x,y\in\Om\\ x\ne y}}
 \left\{
 \frac{|u(x)-u(y)|}{\Fo^{\,\al}(x-y)}
 X_1^{(1-\e)/p}\!
 \Big(\frac{\Fo(x-y)}{D}\Big)
 \right\}
 \le
 C\,\big(I_{p,F}[u]\big)^{1/p},
\]
for some constants $C>0$ and $D\geq\diam_F(\Om)$. Choosing $y=0$ we obtain
\[
 \frac{|u(x)|}{\Fo^{\,\al}(x)}
 X_1^{(1-\e)/p}\!
 \Big(\frac{\Fo(x)}{D}\Big)
 \le
 C\,\big(I_{p,F}[u]\big)^{1/p}
 \quad\forall~x\in\Om\setminus\{0\}.
\]
This readily implies that
\[
  \int_{\Om}\frac{|u|^p}{\Fo^{\,p}}X_1^{2-\e/2}\Big(\frac{\Fo}{D}\Big)~\dd x
   \leq
    C^p\,I_{p,F}[u]\int_{\Om}\Fo^{\,-n}X_1^{1+\e/2}\Big(\frac{\Fo}{D}\Big)~\dd x.
\]
Clearly, since $\e>0$ the integral on the right is a finite constant independent of $u$. Thus \eqref{X2-sharp} would hold with $\gamma=2-\e/2<2$, contradicting the optimality part of Theorem~\ref{pn-thm-X2}.
\end{proof}

\subsection*{Acknowledgements}
G.~di Blasio and G.~Pisante are members of the Gruppo Nazionale per l'Analisi Matematica, la Probabilit\`a e le loro Applicazioni (GNAMPA) of the Istituto Nazionale di Alta Matematica (INdAM), whose support is gratefully acknowledged. Part of this work was carried out while G.~Psaradakis was Visiting Professor at the Department of Mathematics and Physics of the Universit\`a degli Studi della Campania ``Luigi Vanvitelli'' (D.R. 709/2025), whose hospitality and support are gratefully acknowledged.

\subsection*{Disclosure on AI assistance}
The authors used AI-based tools for literature searches,
checking computations and references, and language editing. All results, proofs
and the final text were developed by the authors, who take full responsibility
for the content of the paper.

\bigskip
{\footnotesize
\noindent (G.~di Blasio) \textsc{Dipartimento di Matematica e Fisica, Universit\`a
degli Studi della Campania ``L.~Vanvitelli'', Viale Lincoln 5, 81100 Caserta,
Italy}\par
\noindent\textit{Email address:} \texttt{giuseppina.diblasio@unicampania.it}

\medskip
\noindent (G.~Pisante) \textsc{Dipartimento di Matematica e Fisica, Universit\`a
degli Studi della Campania ``L.~Vanvitelli'', Viale Lincoln 5, 81100 Caserta,
Italy}\par
\noindent\textit{Email address:} \texttt{giovanni.pisante@unicampania.it}

\medskip
\noindent (G.~Psaradakis) \textsc{Department of Mathematics, University of Western
Macedonia, 52100 Kastoria, Greece}\par
\noindent\textit{Email address:} \texttt{gpsaradakis@uowm.gr}
}
\end{document}